\documentclass[12pt, reqno]{amsart}
\usepackage{graphicx} 
\usepackage[utf8]{inputenc}
\usepackage[english]{babel}
\usepackage{amsthm}
\usepackage{pdfpages}
\usepackage{listings}
\definecolor{backcolour}{rgb}{1,1,1}

\usepackage{subcaption}
\usepackage{hyperref}
\usepackage{amssymb}
\usepackage{mathtools,tikz-cd}
\usepackage{latexsym}
\usepackage{parskip}
\usepackage{graphicx}
\graphicspath{{images/}}
\usepackage[inner=1.2in, outer=1in]{geometry}
\usepackage{mathrsfs}
\DeclareMathSymbol\unlhd{\mathrel}{lasy}{"02}
\theoremstyle{definition}
\newtheorem{theorem}{Theorem}
\newtheorem{propn}{Proposition}
\usepackage{chngcntr}
\counterwithin{theorem}{section}
\counterwithin{propn}{section}
\newtheorem{corollary}{Corollary}[theorem]
\counterwithin{corollary}{section}

\theoremstyle{definition}
\newtheorem{example}{Example}
\counterwithin{example}{section}
\theoremstyle{remark}

\theoremstyle{definition}
\newtheorem{definition}{Definition}[section]
\usepackage{filecontents}
\usepackage[nice]{nicefrac}
\usepackage{appendix}
\usepackage{algorithm}
\usepackage{algpseudocode}
\usetikzlibrary{petri,shapes}
\usetikzlibrary{positioning, arrows}

\makeatletter
\let\original@lbibitem\@lbibitem
\def\@lbibitem[#1]#2{%
  \def\currentbibkey{#2}%
  \def\singularbibkey{DGPS}%
  \ifx\currentbibkey\singularbibkey
    \original@lbibitem[BDGPS26]{#2}%
  \else
    \original@lbibitem[#1]{#2}%
  \fi
}
\makeatother

\title{Schreyer Resolutions over the Exterior Algebra}

\author{Janko B\"ohm}
\address[Janko B\"ohm]{Fachbereich Mathematik, RPTU Kaiserslautern-Landau, D-67653 Kaiserslautern}
\email{boehm@mathematik.uni-kl.de}
\thanks{{Gef\"ordert durch die Deutsche Forschungsgemeinschaft (DFG) - Projektnummer 286237555 - TRR 195 (Funded by the Deutsche Forschungsgemeinschaft (DFG, German Research Foundation) - Project-ID 286237555 - TRR 195). The work of JB was supported by Project B5 of SFB-TRR 195 and Potentialbereich \emph{SymbTools - Symbolic Tools in Mathematics and their Application} of the Forschungsinitiative Rheinland-Pfalz.}} 

\author{Lakshmi Ramesh}
\address[Lakshmi Ramesh]{Fakult\"at f\"ur Mathematik, Universit\"at Bielefeld, D-33615 Bielefeld}
\email{lramesh@math.uni-bielefeld.de}

\begin{document}

\begin{abstract}

    Schreyer's algorithm is usually the fastest way to determine a (typically non-minimal) free resolution of a finitely presented module over a polynomial ring. We adapt the refined Schreyer algorithm from \cite{erocal2016refined} to compute free resolutions over the exterior algebra, relying on relative Gr\"obner bases. We illustrate the use of our algorithm to compute the cohomology of coherent sheaves over projective space, which by the BGG correspondence and \cite{eisenbud2003sheaf} can be computed via free resolutions over the exterior algebra. Schreyer's method relies on a tree traversal and thus has the potential for parallel computations. We report on ongoing work on a massively parallel implementation, observing that Gnawali's parallel approach over polynomial rings \cite{gnawali2024massively} can be carried over to our setting.
\end{abstract}

\maketitle

\section{Introduction}\label{section:intro}

In this paper, we develop an algorithm to compute Schreyer-type free resolutions of modules over the exterior algebra. Our approach is based on the refined Schreyer algorithm, developed by Er\"ocal, Motsak, Schreyer and Steenpass in \cite{erocal2016refined}. Our algorithm is implemented in the library \texttt{sresext.lib} of the computer algebra system \textsc{Singular} \cite{DGPS}.\footnote{The library is contained in the latest release of \cite{DGPS} as well as, in its latest version, in the GitHub repository of \textsc{Singular}, see \url{https://github.com/Singular}.}  Our motivation arises from the Bernstein-Gel'fand-Gel'fand correspondence, which allows one to compute sheaf cohomology of coherent sheaves  over projective spaces via free resolutions over the exterior algebra. A further motivation is that Schreyer-type algorithms to determine syzygies and resolutions have, in contrast to more classical methods based on Gröbner basis techniques to find syzygies, a significant potential for parallelisation. In the commutative case, massively-parallel free resolution algorithms have been implemented in \cite{G25, gnawali2024massively}. This implementation relies on the \textsc{Singular}/\textsc{GPI-Space} framework \cite{BDPRR21} as its technical foundation, and uses Petri-nets for coordination. While our current implementation is just a demonstrator for our algorithm, a powerful \textsc{Singular}-kernel based implementation is currently under development. It will also be the basis for the computational units in a high-performance massively-parallel version. In Section~\ref{section:parallel}, we show that existing parallel techniques for the commutative case can be extended to our non-commutative setup. 

The exterior algebra $E$ of a vector space $V$ can be defined as follows. Let $K$ be an algebraically closed field, and $K\langle V \rangle $ be the free associative algebra generated by an ordered basis $\{x_1, \dots, x_n \}$ of $V$. We then define the $K$-algebra $\mathcal N$ as the quotient $\nicefrac{K \langle V \rangle}{J}$ where $J$ is the two-sided ideal of $K\langle V \rangle$ generated by $\{x_i x_j + x_j x_i | 1 \leq i < j \leq n \}$.  $\mathcal {N}$ is a ring with no non-trivial zero divisors, that has a PBW (Poincar\'e-Birkhoff-Witt) basis, and is sometimes called a PBW algebra. 
We describe $\mathcal N$ as anti-commutative (if $\operatorname{char}(K)\neq 2$, it is commutative if $\operatorname{char}(K) = 2$) due to the property that the order of multiplication affects the sign of the product. That is, $x_i x_j = -x_jx_i$ for $i \neq j$. Taking the quotient of $\mathcal N$ by the ideal $I$ generated by $\{ x_i^2 | 1 \leq i \leq n \}$, we obtain the exterior algebra (of an $n$-dimensional $K$-vector space) $E = \nicefrac{\mathcal N}{I}$. In fact, this is how exterior algebras are constructed in \textsc{Singular}.

In \cite{erocal2016refined}, the authors develop a refined Schreyer algorithm, which computes syzygy modules and free resolutions over the polynomial ring. It is ``refined" in the following sense. The algorithm starts with the Schreyer frame, which is the sequence of module morphisms of lead modules associated to a Schreyer resolution \cite{la1998strategies}, the tail terms are determined as sums of terms corresponding to vertices of trees rooted in the lead terms. One of the main improvements of the algorithm in \cite{erocal2016refined} is that it avoids computation of so-called lower order terms that do not contribute to syzygies. This approach is thus significantly faster than the original Schreyer algorithm and is the current state-of-the art for computing (typically non-minimal) free resolutions over polynomial rings. An efficient implementation of the algorithm (in the commutative setting) is available in \textsc{Singular}. We adapt this approach to the setting of modules over the exterior algebra.

One of the primary applications of the computation of free resolutions over the exterior algebra lies in the computation of sheaf cohomology. %
This is due to Eisenbud, Fl{\o}ystad and Schreyer \cite{eisenbud2003sheaf}, who showed that one can compute sheaf cohomology of coherent sheaves over projective space by computing free resolutions of modules over the exterior algebra. This is true due to a correspondence between the derived category of bounded complexes of coherent sheaves, and that of finitely generated modules over the exterior algebra, called the BGG correspondence, as described in the work of Bernstein, Gel'fand and Gel'fand in \cite{bgg78}. 
In this paper, we provide examples of cohomology computations based on our resolution algorithm to illustrate this.

In our approach, the refined Schreyer algorithm is adapted to the anti-commutative setting of $\mathcal N$.
In combination with the theory of relative Gr\"obner bases, we compute the syzygy $E$-module by first computing the syzygy $\mathcal N$-module of the corresponding lifted module. The set of canonical representatives of the $E$-Gr\"obner basis in the lifted $\mathcal N$-module is a Gr\"obner basis relative to the ideal $I = \langle x_i^2 \mid 1 \leq i \leq n \rangle$. This fact allows us to construct a Gr\"obner basis of the lifted module. The syzygy computation is performed over $\mathcal N$ using the adapted refined Schreyer algorithm. Making use of relative Gr\"obner basis methods for modules over quotient rings, we can finally obtain an $E$-Gr\"obner basis for the syzygy $E$-module lying below the syzygy $\mathcal N$-module. 

In Section \ref{section:nc-fundamentals}, we introduce our notation and conventions for the anti-commutative setting. We also adapt Schreyer's theorem to syzygies over $\mathcal{N}$.
In Section \ref{section:syz-ext}, we apply relative Gr\"obner basis methods to address resolutions over the exterior algebra.
In Section \ref{section:schreyer-algorithm}, we generalize the refined Schreyer resolution algorithm first to $\mathcal{N}$, and then to the exterior algebra setting.
In Section \ref{bgg} we discuss the application of the Schreyer resolution algorithm over the exterior algebra to computing sheaf cohomology. In particular, as a tool in the cohomology algorithm, we discuss the efficient minimization of the Schreyer resolution.
In Section \ref{section:parallel}, we discuss the feasibility of a parallel implementation.

\textbf{Acknowledgements.} We are grateful for the technical help from Dr. Hans Sch\"onemann in the implementation of the algorithm in \texttt{Singular}. We also thank Dr. Santosh Gnawali for his help in conceptualising the parallel algorithm. 

\section{Non-Commutative Fundamentals}\label{section:nc-fundamentals}

\subsection{Gr\"obner Bases over $\mathcal N$}\label{groebnerbases}

We define monomial orderings on the PBW algebra $\mathcal{N}$, which will induce orderings on $E$. We call $\{ x_1^{\alpha_1} \dots x_n^{\alpha_n}  | \alpha_i \in \mathbb Z_{\geq 0} \forall i \}$, where $x_1, \dots, x_n$ are in a (chosen) lexicographic order, the set of standard monomials in $\mathcal N$. For an element $f \in \mathcal N$, we denote the lead monomial of $f$ written in standard form by $\operatorname{LM}(f)$ (which is now simply referred to as the lead monomial), the lead coefficient of $f$ by $\operatorname{LC}(f)$, and the lead term of $f$ by $\operatorname{LT}(f)$. We use the following definition of monomial orderings over non-commutative rings from \cite[II. Def. 1.2]{li2007noncommutative}.

For exponent vectors $\alpha, \gamma$, we write $x^\alpha x^\gamma = \varepsilon (\alpha, \gamma)x^{\alpha + \gamma}$, with $\varepsilon(\alpha, \gamma) = (-1)^{\sum_{p>q}\alpha_p \gamma_q}$.
\begin{definition}\label{def:non-commutative-ring}
    Let $>$ be a well ordering on the set of standard monomials of $\mathcal N$. $>$ is a monomial ordering on $\mathcal N$ if it satisfies:
    \begin{itemize}
        \item For any non-zero monomials  $w,v,s,u$ in $\mathcal N$ with $u > w$ and $\operatorname{LM}(vws) \neq 0$, $\operatorname{LM}(vus) \neq 0$, then $\operatorname{LM}(vus) > \operatorname{LM}(vws)$.
        \item For any non-zero monomials $w,v,s,u$ in $\mathcal N$, if $u = \operatorname{LM}(vws)$ where $v \neq 1$ or $s \neq 1$, then $u > w$. (Hence $u>1$ for all non-zero monomials $u \neq 1$.) 
        \item $w$ is said to be left divisible by $u$, that is $u|w$, if $w = \operatorname{LM}(vu)$ for some monomials $v \in \mathcal N$. 
    \end{itemize}
\end{definition}

This is compatible with monomial orderings on the polynomial ring as the order of multiplication only affects the coefficient. Hence, a monomial ordering on the polynomial ring is also a monomial ordering on $\mathcal N$. 
A monomial ordering $>_m$ on an $\mathcal N$-module must be, as in the commutative case, compatible with an ordering $>$ on $\mathcal N$. 
As in the commutative case, the monomial ordering may give preference to the monomial or the component. 

In this text, we consider left modules over $\mathcal N$ that are finitely generated. That is, a module $M$ has a finite left-basis $m_1, \dots, m_r$ such that any element $m \in M$ can be written as $m = \sum_{i=1}^r u_i m_i$, $u_i \in \mathcal N$. 
Since $\mathcal N$ is a left- and right-Noetherian ring, it follows that $M$ is also finitely presented and hence has finite Gr\"obner bases, see \cite[Prop. 2.1]{MR3243257}.
We use the following definition of Gr\"obner bases of modules over non-commutative rings, which is an extension of the definition in \cite[II. Def. 3.2]{li2007noncommutative}.
\begin{definition}\label{def:groebner-basis}
    Let $M$ be a left $\mathcal N$-module generated by a finite subset $\mathcal G \subset M$. If every $p \in M$ has a presentation $$p = \sum_i  u_i g_i \; \; \; \; u_i\in \mathcal N,\; g_i \in \mathcal G,$$ with $\operatorname{LM}(p) \geq_m \operatorname{LM}(u_i g_i)$, then $\mathcal G$ is a left Gr\"obner basis of $M$ with respect to $>_m$. We call a presentation of an element $p$ as an $\mathcal N$ linear sum of elements of $\mathcal G$ as a standard presentation of $p$ with respect to $\mathcal G$. It is not necessarily unique. 
\end{definition}
Let $s = a x^\alpha e_{i}$ and $t = b x^\gamma e_{j}$ be two terms in $M$ for $a,b \in K$ such that $\alpha \leq \gamma$. Then, we define the left quotient of $t$ by $s$ as 
\begin{align*}
    \operatorname{lquot}(t,s) = \frac{b}{a \varepsilon(\gamma - \alpha, \alpha)}x^{\gamma - \alpha}.
\end{align*}

The normal form of an element in an associative algebra is defined by the reduced division algorithm described in \cite[Section II.2]{li2007noncommutative}, see Algorithm \ref{alg:n-division}. Let $M\subset F=\mathcal N^m$ and let $\mathcal G$ be a Gr\"obner basis of $M$. By \cite[II. Prop. 3.3]{li2007noncommutative}, this reduced division defines the unique normal form map
\[
\operatorname{NF}(-,\mathcal G)\colon F\longrightarrow
\operatorname{span}_K\{x^\alpha e_j\mid \operatorname{LM}(g)\nmid x^\alpha e_j\text{ for all }g\in\mathcal G\}.
\]
However, the standard presentation is not unique. It is clear from Algorithm \ref{alg:n-division} and Definition \ref{def:groebner-basis} that $\mathcal G$ is a Gr\"obner basis of $M$ if and only if for every $p \in M$, $\operatorname{NF}(p, \mathcal G) = 0$.

\begin{algorithm}[ht]
\caption{Reduced division of $p$ by $\mathcal G$}\label{alg:n-division}
\textbf{Input:} A Gr\"obner basis $\mathcal G = \{f_1, \dots, f_r\}$ of $M$ with respect to $>_m$, an element $p$ in ambient free module.\\
\textbf{Output:} A tuple $(q,r_0)$ where $p = q + r_0 $ and $q= \sum_{i=1}^r \alpha_i f_i , r_0 = \operatorname{NF}(p,\mathcal G)$ for $\alpha_i \in \mathcal N$ such that $\operatorname{LM}(p) \geq_m \operatorname{LM}(\alpha_i f_i)$ and no term of $\operatorname{NF}(p, \mathcal G )$ is divisible by $\operatorname{LM}(f_i)$ for any $1 \leq i \leq r$.
\begin{algorithmic}[1]
\State $g = p$, $r_0 = 0$, $q=0$
\While {$g \neq 0$}
    \If {$\text{LM}(f_i) \mid \operatorname{LM}(g)$ for some $f_i \in \mathcal G$}
        \State $a = \operatorname{lquot}(\operatorname{LT}(g),\operatorname{LT}(f_i))$
        \State $q += a f_i$
        \State $g+= - a f_i$
    \Else 
        \State $r_0 += \operatorname{LT}(g)$
        \State $g += -\operatorname{LT}(g)$
    \EndIf
\EndWhile
\State \textbf{return} $(q, r_0)$
\end{algorithmic}
\end{algorithm}

We also introduce some conventions required for the algorithm in \cite{erocal2016refined}. 

\begin{definition}\label{def:nomenclature}
Let $F := \mathcal N^r$ be a left free $\mathcal N$-module of rank $r$. Let $e_1, \dots e_r$ be a basis of $F$.  
    \begin{enumerate}
        \item A monomial in $F$ is a product of a monomial $m$ in $\mathcal N$ with a basis element $e_i$, $1 \leq i \leq r$.
        \item A term in $F$ is a monomial (as defined above), with a coefficient from $K^*$. 
        \item The least common multiple of two monomials $f_i = m_1 e_i$ and $f_j = m_2 e_j$ is defined as
        $$ \operatorname{LCM}(f_i, f_j) = \begin{cases}
            \operatorname{LCM} (m_1, m_2) e_i \quad \; \; \operatorname{if }i = j\\
            0 \qquad \qquad \qquad \qquad \text{otherwise.}
        \end{cases} $$
        The LCM of $m_1, m_2$ is written as per the convention in Section \ref{groebnerbases}, as a monomial in standard form. That is, the sign of the resulting anti-commutative product is ignored. 
        \item A monomial $f_i = m_1 e_i$ is said to divide another monomial $f_j = m_2 e_j$ if $i = j$ and if $m_1$ divides $m_2$ (as defined in Definition \ref{def:non-commutative-ring}). The quotient lies in $\mathcal N$.
    \end{enumerate}
\end{definition}

These definitions may be appropriately extended to submodules of free modules. 

Let $M$ be a submodule of $F$, a left free module of $\mathcal N$. We denote the left module generated by the lead monomials of all non-zero elements of $M$ as $ \operatorname{L}(M)$. This is called the\textit{ (left) lead module.}
By Definition~\ref{def:groebner-basis} and \cite[II. Prop. 4.2]{li2007noncommutative}, $\mathcal G$ is a left Gr\"obner basis of $M$ if and only if the lead module $\operatorname{L}(M)$ is generated by the lead monomials $\{\operatorname{LM}(g_i) \mid g_i \in \mathcal G \}$. We denote the left module generated by a set $S$ as $\langle S]$.

\subsection{Syzygy Modules over $\mathcal N$} 

\begin{definition}
    Let $\mathcal G = \{f_1, \dots, f_r\}$ be a set of generators of a finitely generated left $\mathcal N$-module $M$. Then, a left-syzygy of $M$ is an element $(m_1, \dots, m_r) \in \mathcal N^r$, such that $\sum_{i=1}^r m_i f_i = 0$. The set of all syzygies of $M$ forms a submodule of $\mathcal N^r$, called the syzygy module $\operatorname{Syz}_\mathcal N(\mathcal G)$ of $\mathcal G$.
\end{definition}
Consider the morphism $\phi: \mathcal N^r \rightarrow M \subset \mathcal N^m$, where $\phi(e_i) = f_i$. Then, $\operatorname{Syz}_{\mathcal N}(\mathcal G) = \operatorname{ker}(\phi)$.
As we will see from Theorem \ref{schreyers-theorem}, if we start with generators $f_1, \dots, f_r$ of $M$ that form a Gr\"obner basis, we can compute a Gr\"obner basis of $\operatorname{Syz}_{\mathcal N}(\mathcal G)$
with respect to the Schreyer ordering.

\begin{definition}
Let $\mathcal G=(f_{1},\ldots,f_{r})$ be an ordered set of elements in $\mathcal N^{m}$ and $>_m$ a monomial ordering on $\mathcal N^{m}$. The monomial ordering $\succ_\mathcal{G}$ on $\mathcal N^{r}$
defined by%
$$x^{\alpha}e_{i}\succ_\mathcal{G} x^{\beta}e_{j}\Longleftrightarrow \operatorname{LM}(x^{\alpha}f_{i})>_m \operatorname{LM}(x^{\beta}f_{j})\text{ or }\left(  \operatorname{LM}(x^{\alpha}f_{i})=\operatorname{LM}(x^{\beta}%
f_{j})\text{ and }i<j\right)$$
is called the \textit{Schreyer ordering} on $\mathcal N^r$ induced by $>_m$ and $\mathcal G$. 
\end{definition}

Given a Gr\"obner basis of $M$, we can define the S-vectors, or syzygy vectors, from which we construct the Gr\"obner basis of the syzygy module. Let $\mathcal G = \{ f_1, \dots, f_r\}$ be a Gr\"obner basis of $M$ with respect to $>_m$. 
Then for every $i,j \in \{ 1, \dots , r \}$ with $\operatorname{LCM}(\operatorname{LM}(f_j), \operatorname{LM}(f_i))\neq 0$ write 
$$ m_{ij} := \operatorname{lquot}(\operatorname{LCM} ( \operatorname{LM}(f_j), \operatorname{LM} ( f_i )) , \operatorname{LT}(f_j)).$$

We then can define the S-vector of $f_i, f_j$ by 
$$\operatorname{S-vector}(f_i, f_j)= m_{ji}f_i -m_{ij}f_j.$$
The aim of the S-vector is to cancel the lead terms of $f_i$ and $f_j$. 
Note that, when all monomials of the S-vector are written in standard form, the lead terms indeed cancel.
We can now state the anti-commutative Schreyer theorem, which constructively shows us the existence of a Gr\"obner basis of the syzygy module $\operatorname{Syz}_{\mathcal N}(\mathcal G)$.

\begin{theorem}[Anti-commutative Schreyer Theorem]\label{schreyers-theorem}
    Let $\mathcal G = \{ f_1 , \dots , f_r \} \subset \mathcal N^m$ be a Gr\"obner basis of a finitely generated left $\mathcal N$-module $M$ with respect to a module monomial ordering $>_m$ on $\mathcal N^m$. Then for each pair $(f_i, f_j)$ where $\operatorname{LCM}(\operatorname{LM}(f_j), \operatorname{LM}(f_i)) \neq 0$ and $1 \leq i < j \leq r$, let 
    $$\operatorname{S-vector}(f_i, f_j) = \sum_{k=1}^r g_k^{(ij)} f_k$$
    be a standard presentation of $\operatorname{S-vector}(f_i, f_j)$ with respect to $\mathcal G$. Denote by $e_1, \dots , e_r$ the canonical basis of $\mathcal N^r$. 
    Then
    $$\operatorname{s}(f_i,f_j) : = m_{ji}e_i - m_{ij}e_j - (\sum_{k=1}^r g_k^{(ij)} e_k)$$
    for all $i<j \in \{1, \dots r\}$ with $\operatorname{LCM}(\operatorname{LM}(f_j), \operatorname{LM}(f_i)) \neq 0$
    form a Gr\"obner basis $\mathcal G'$ of $\operatorname{Syz}_{\mathcal N}(\mathcal G)$ with respect to the Schreyer ordering $\succ_\mathcal G$ induced by $>_{m}$ and $\mathcal G$. 
\end{theorem}

This is a specific case of a more general non-commutative Schreyer theorem for modules over solvable polynomial algebras, see \cite[Thm. 4.8]{Levandovskyy2005}.

\section{Computing Syzygies over the Exterior Algebra}\label{section:syz-ext}

Syzygy computation using the refined Schreyer Algorithm as described in Section \ref{section:schreyer-algorithm} requires both a Gr\"obner basis and the division algorithm, which are described for modules over PBW algebras, and therefore must be adapted to the case of modules over a quotient ring of a PBW algebra. For this we use the theory of relative Gr\"obner bases. This requires that we lift the given $E$-module $\Bar M \subset E^m$ to a $\mathcal N$-module $M \subset \mathcal N^m$, in order to compute the syzygy, see \eqref{diagram-lifting}, the corresponding presentation maps are as follows.

Recall that we define $E = \nicefrac{\mathcal N}{I}$ for 
\begin{align}\label{eq:I}
    I = \langle x_1^2, \dots, x_n^2 \rangle .
\end{align}
Let $\pi_m:\mathcal N^m \rightarrow E^m$ be the component-wise projection map and let $M = \pi^{-1}_m(\Bar{M})$. Let $B = (f_1, \dots, f_r)$ be an ordered set of elements in $M$ such that $\{\pi_m(f_1), \dots, \pi_m(f_r) \}$ generates $\Bar{M}$. Moreover, we order the standard generators of $I^m$ as 
\begin{align}\label{eq:gs}
    (b_1, \dots, b_{nm}) = (x_{i}^2 e_j \mid 1\leq i \leq n, 1 \leq j \leq m).
\end{align}
Let $\phi_{\mathcal N}$ and $\phi_E$ be the presentation maps $\phi_{\mathcal N}(u,v) = \sum_{i=1}^r u_i f_i + \sum_{j = 1}^{nm} v_j b_j$ and $\phi_E(\pi_r(u)) = \sum_{i=1}^r [u_i] \pi_m(f_i)$ respectively, and $\rho(u,v) = ([u_1], \dots, [u_r])$. Then, by Theorems~\ref{theorem:relative-gb} and \ref{thm:nsyztoesyz},
\begin{equation}\label{diagram-lifting}
\begin{tikzcd}
\mathcal N^{r + nm} \arrow[r, "\phi_{\mathcal N}" ] \arrow[d, "\rho"] & \mathcal N^m \arrow[d, "\pi_m"]\\
E^{r}  \arrow[r, "\phi_{E}"] & E^m
\end{tikzcd}
\end{equation}
commutes, and $\operatorname{ker}(\phi_{E})$ is the desired syzygy module.
In our method, a free resolution over the exterior algebra is computed by computing each syzygy over the lifted $\mathcal N$-module, and then projecting to the corresponding $E$-module. Then, the next syzygy is computed by again lifting the module, and so on.
At the outset of the resolution computation, \textsc{Singular} computes a Gr\"obner basis of the $E$-module. In this context, we wish to find the corresponding Gr\"obner basis of the lifted $\mathcal N$-module $M$, whose syzygies we can suitably project to obtain a Gr\"obner basis of the syzygy $E$-module. 

\subsection{Relative Gr\"obner Basis}\label{relativegb}

As discussed in Section \ref{groebnerbases}, the ring $\mathcal N$ can be equipped with a monomial ordering $>$. We want to define the corresponding monomial ordering $>_E$ on the quotient ring $E = \nicefrac{\mathcal N}{I}$. Consider the projection map $\pi: \mathcal N \rightarrow E$. Let $\sigma: E \rightarrow \mathcal N$ be its $K$-linear section taking each exterior monomial to its unique square-free
canonical representative. Then, 
\begin{gather*}
    m_1 >_E m_2 \iff \sigma(m_1) > \sigma(m_2).
\end{gather*}

These notions of projections and monomial orderings can be extended to the module case. Consider the free module $\mathcal N^m$. We fix the standard basis $\{ e_1, \dots ,e_m\}$ and a module monomial ordering $>_m$ that agrees with the ordering $>$ on $\mathcal N$. Then the projection map $\pi_m: \mathcal N^m \rightarrow E^m$ gives us a corresponding module monomial ordering $>_{Em}$ on $E^m$. Denoting $\sigma_m$ as the component-wise square-free representative in $\mathcal N$, we have that
\begin{gather}\label{monorderNE}
    m_i e_i >_{Em} m_j e_j \iff \sigma_m(m_i e_i) >_m \sigma_m(m_j e_j).
\end{gather}
Therefore given any $E$-module $\Bar M$ equipped with a module monomial ordering $>_{Em}$, there exists a corresponding $\mathcal N$-module $M := \pi_m^{-1}(\Bar M)$ with a corresponding module monomial ordering $>_m$ as in \eqref{monorderNE}.

In order to be able to define Gr\"obner bases in the setting of modules over quotient rings, we require a division algorithm that takes into account the quotient ideal $I$. 
Given a submodule $H$ and a Gr\"obner basis $\mathcal H = \{f_1, \dots , f_r\}$ of $H$, we wish to divide a vector $v$ by $\mathcal H$, where $\operatorname{NF}_{\mathcal H}(f_i) = f_i$ for all $i \in \{1, \dots , r\}$. We obtain a vector $\Bar{v}$ such that $\operatorname{NF}_{\mathcal H}(\Bar v) = \Bar v$ and $v - \sum_{i=1}^r q_i f_i - \Bar v \in  H$ such that no monomial of $\bar v$ is divisible by any lead monomial of any $f_k$ or $x_{i}^2e_j$ for any $i\in[n], j\in[m], k\in[r]$. We can do this algorithmically with Algorithm \ref{alg:relative-division} (as described in \cite[Alg. 1]{hashemi2021relative}). 

\begin{algorithm}[ht]
\caption{Relative Division}\label{alg:relative-division}
\textbf{Input:} A Gr\"obner basis $\mathcal H$ of $H$, a set of vectors $\{f_1, \dots , f_r\}$ that are all reduced with respect to $\mathcal H$ and a vector $v$.\\
\textbf{Output:} A vector $\Bar v$ such that no term of $\Bar v$ is divisible by $\operatorname{LM}(f_i)$ for any $f_i \in \{f_1, \dots , f_r\}$ or by any term in $\langle \operatorname{L}(H)]$.
\begin{algorithmic}[1]
\State $x := v$, $\Bar v := 0$
\For {$i = 1, \dots , r$}
    \State $q_i := 0$
\EndFor
\While {$x \neq 0$}
    \If {$\operatorname{LM}(x)\; \in \; \operatorname{L}(H)$}
        \State Choose $g\; \in \; \mathcal H$ such that $\operatorname{LM}(g) | \operatorname{LM}(x)$
        \State $x := x - \operatorname{lquot}(\operatorname{LT}(x),\operatorname{LT}(g))g$
    \ElsIf {$\operatorname{LM}(x) \; \in \; \langle\{ \operatorname{LM}(f_1) , \dots , \operatorname{LM}(f_r) \}]$}
        \State Choose $f_i$ such that $\operatorname{LM}(f_i) |  \operatorname{LM}(x)$
        \State $q_i := q_i + \operatorname{lquot}(\operatorname{LT}(x),\operatorname{LT}(f_i))$
        \State $x := x - \operatorname{lquot}(\operatorname{LT}(x),\operatorname{LT}(f_i))f_i$
    \Else 
        \State $\Bar v : = \Bar v + \operatorname{LT}(x)$
        \State $x : = x - \operatorname{LT}(x)$
    \EndIf
\EndWhile
\State \textbf{return} $\bar v$
\end{algorithmic}
\end{algorithm}

We can now tackle the issue of the correct Gr\"obner basis of $M$, given a Gr\"obner basis $\Bar B \subset \Bar M$, using \cite[Def. 3.7]{hashemi2021relative}.

\begin{definition}\label{def:relative_gb}
    Let $I$ be a two-sided ideal of $\mathcal N$ and let $M \subset \mathcal N^m$ be a finitely generated left $\mathcal N$-submodule of a free module containing $I^m$, and let $\mathcal I$ be a Gr\"obner basis of $I^m$. 
    \begin{itemize}
        \item A finite set $B \subset M$ such that $B \cap I^m = \emptyset$ is called a \textit{Gr\"obner basis of M relative to $I$} if 
        \begin{gather*}
            \langle \{\operatorname{LM}( f_i ) \mid f_i \in B\}]  + \operatorname{L}(\mathcal I) =  \operatorname{L}(M).
        \end{gather*}
        \item A finite set of coordinate-wise projections $\{ [f_1], \dots , [f_r]\} \subset \Bar M$ is a Gr\"obner basis of the $\nicefrac{\mathcal N}{I}$-module $\Bar M$ with respect to $>_{\Bar M}$ if the set of canonical lifts $\{ f_1, \dots f_r \}$ is a Gr\"obner basis of the $\mathcal N$-module $M$ relative to $ I$ with respect to $>_M$. 
    \end{itemize}
\end{definition}

The following theorem, see \cite[Prop. 3.4,3.6]{hashemi2021relative}, gives us the Gr\"obner basis of $M$ that will be the input of the adapted version of the Schreyer algorithm.

\begin{theorem}\cite{hashemi2021relative}\label{theorem:relative-gb}
    Let $I$ be a two-sided ideal of $\mathcal N$ and let $M \subset \mathcal N^m$ be a left $\mathcal N$-submodule of a free module, containing $I^m$ and let $\{ h_1, \dots , h_k\}$ be a Gr\"obner basis of $I$. Let  $\{e_1, \dots e_m\}$ be the standard basis of $\mathcal N^m$ and let $B$ be a finite set contained in $M$ such that $B \cap I^m = \emptyset$. Then the following are equivalent:
    \begin{enumerate}
        \item $B$ is a Gr\"obner basis of $M$ relative to $I$.
        \item $B \cup \{ h_i e_j \; | \; 1 \leq i \leq k, 1 \leq j \leq m \}$ is a Gr\"obner basis of $M$.
    \end{enumerate}
\end{theorem}

Note that for our ideal $I$ under consideration as defined in \eqref{eq:I}, $\{x_i^2 e_j \; | \; 1\leq i \leq n, 1\leq j\leq m\}$ is a Gr\"obner basis of $I^m$. From the above result, we have that if $\{f_1, \dots, f_r \} \cap I^m = \emptyset$, $\{ \pi_m(f_1), \dots , \pi_m(f_r) \}$ is a left Gr\"obner basis of the $E$-module $\Bar M$ if and only if $$\{f_1,\dots f_r\} \cup \{x_i^2 e_j \; | \; 1\leq i \leq n, 1\leq j\leq m\}$$ is a left Gr\"obner basis of the $\mathcal N$-module $M$.

\subsection{From $\mathcal N$-syzygy to $E$-syzygy}\label{quotient}

We order the lifted Gr\"obner basis of $M$ as $$\mathcal G_M = (f_1, \dots, f_r, b_1, \dots, b_{nm}),$$
where $b_1, \dots, b_{nm}$ are the ordered generators of $I^m$, see \eqref{eq:gs}.
By Theorem~\ref{schreyers-theorem}, the result of the syzygy computation is the Gr\"obner basis of 
$\operatorname{Syz}_{\mathcal N}(\mathcal G_M)\subset \mathcal N^p$ where $p = r + nm$. 
The following theorem, applied to the PBW case, is the required result to obtain the Gr\"obner basis of $\operatorname{Syz}_E(\mathcal G_E)$, where $\mathcal G_E = \{\pi_m(f_1), \dots \pi_m(f_r) \}$.

\begin{theorem}\cite{hashemi2021relative}\label{thm:nsyztoesyz}
    The non-zero images of the Schreyer syzygies $\operatorname{Syz}_{\mathcal N}(\mathcal G_M)$ under the map 
    \begin{equation}\label{eqn:quotient}
\begin{aligned}
    \rho:  \mathcal N^p \rightarrow E^r\\
    (v_1, \dots , v_r, v_{r+1}, \dots , v_p) \mapsto & ([v_1], \dots, [v_r]).
\end{aligned}
\end{equation}
form a Gr\"obner basis of the module $\operatorname{Syz}_E(\mathcal G_E)$ with respect to the Schreyer ordering $\succ_{\mathcal G_E}$.
\end{theorem}

\section{Refined Schreyer Algorithm}\label{section:schreyer-algorithm}
In this section, we discuss the refined Schreyer Algorithm proposed by Er{\"o}cal, Motsak, Schreyer and Steenpass in \cite{erocal2016refined} to compute the Gr\"obner basis of a syzygy module. This algorithm is a revised and optimised formulation of the original syzygy algorithm given by Schreyer in \cite{schreyer-thesis}.
Er{\"o}cal, Motsak, Schreyer and Steenpass have devised this algorithm for modules over polynomial rings. We adapt their algorithm to $\mathcal N$-modules. We follow the definitions for terms, monomials and S-vectors, as well as the notation for lead monomials and lead terms as discussed in Section \ref{section:nc-fundamentals}. Recall that the S-vectors over $\mathcal N$ are slightly different than those over the polynomial ring, see Subsection \ref{groebnerbases}. The adaptation of the refined Schreyer algorithm calls for writing all monomials in standard form, in order to avoid sign confusion. We use the induced Schreyer ordering 
\begin{align*}
    me_i \succ m^\prime e_j \iff
    \begin{cases}
\operatorname{LM}(m f_i)>_m\operatorname{LM}(m'f_j),&\text{or}\\
\operatorname{LM}(m f_i)=\operatorname{LM}(m'f_j)\text{ and }i<j.
\end{cases}
\end{align*}

It is important to note that in this section, the word `lift' has a specific connotation, different from that in the previous section. The precise definition is Definition \ref{def:lift}.

\subsection{The Schreyer Frame}

Let $M \subset \mathcal N^m$ be a finitely generated left $\mathcal N$-module, and let $\mathcal G = (f_1, \dots , f_r)$ be an ordered Gr\"obner basis of $M$ with respect to $>_m$. Recall that computing the syzygies of $\mathcal G$ involves computing S-vectors and subtracting their standard presentations, 
\begin{gather*}
    m_{ji} f_i - m_{ij}f_j - \sum_{k = 1}^rg_k^{(ij)}f_k,
\end{gather*}
where $\operatorname{LM}(m_{ji}f_i) = \operatorname{LM}(m_{ij}f_j) > \operatorname{LM}(g_k^{(ij)}f_k)$ for all $1 \leq k \leq r$. Then by applying Schreyer's Theorem \ref{schreyers-theorem}, we obtain the syzygy $m_{ji} e_i - m_{ij}e_j - \sum_{k=1}^rg_k^{(ij)}e_k$. Note that if $i<j$, then the lead monomial of this syzygy is $\operatorname{LM}(m_{ji})e_i$. For any Gr\"obner basis $\mathcal G$ of $M$, $\operatorname{L}_{\succ_{\mathcal G}}(\operatorname{Syz}(\mathcal G))$ is generated by all $m_{ji}e_i$, for $1 \leq j < i \leq r$.
We use this information to compute the lead module of the syzygy module $\operatorname{L}(\operatorname{Syz}(\mathcal G))$. A computation of the lead module gives us a starting point that guides the computation of the entire syzygy module. We now introduce the concept of a Schreyer frame, as defined in \cite[Def. 3.2 - 3.4]{la1998strategies}. 

\begin{definition}\label{def:schreyer-frame}
    \begin{enumerate}
        \item Let $M = F_0 / N$ be a presentation of the finitely generated left $\mathcal N$-module $M$. A Schreyer resolution of $M$ is an exact sequence 
        \begin{gather*}
            \Phi : \dots \xrightarrow[]{\phi_n}  F_{n-1} \xrightarrow[]{\phi_{n-1}} \dots \xrightarrow[]{\phi_1} F_0
        \end{gather*}
        with respect to a Schreyer ordering such that $\text{coker}(\phi_1) = M$ and $\phi_i(e_i)$ is a Gr\"obner basis of $\operatorname{im}(\phi_i)$ with respect to the Schreyer ordering, and $e_i$ is a canonical basis of $F_{i}$
        \item A Schreyer frame is a sequence of free module morphisms
        \begin{gather*}
            \Xi : \dots \xrightarrow[]{\xi_n}  F_{n-1} \xrightarrow[]{\xi_{n-1}} \dots \xrightarrow[]{\xi_1} F_0
        \end{gather*}
        with an ordering $>_m$ on $R^s$ such that $\xi(e_i)$ is a monomial in $F_k$ for any $e_i$ in the canonical basis of $F_{k+1}$, and 
        \begin{itemize}
            \item $\operatorname{LM}(\xi_1(e_i))$ for all $e_i$ in the canonical basis of $F_1$ is a minimal generating set of $\operatorname{L}(N)$.
            \item $\operatorname{LM}(\xi_i(e_j))$ for all $e_j$ in the canonical basis of $F_k$ is a minimal generating set of $\operatorname{L}(\operatorname{ker}(\xi_{i-1}))$.
        \end{itemize}
    \end{enumerate}
\end{definition}

 We first compute a Schreyer frame, and then compute the tail to obtain the exact resolution. The following statement adapted from \cite[Prop. 3.6]{la1998strategies}, with Proposition~\ref{prop:proof-of-correctness}, shows that this results in the desired resolution.
\begin{propn}\cite{la1998strategies}
    For any Schreyer frame $\Xi$ of an $\mathcal N$-module $M$, there exists a Schreyer resolution $\Phi$ of $M$ such that $\Xi = \operatorname{in}(\Phi)$, where $\operatorname{in}(\Phi)$ denotes the sequence of module morphisms 
    \begin{align*}
    \xi_i :& F_i \rightarrow F_{i-1}\\
        &e_j \mapsto \operatorname{LM}(\phi_i(e_j)).
    \end{align*}
\end{propn}

The refined Schreyer algorithm relies on Schreyer frame. Given a Gr\"obner basis $\mathcal G$ of $M$, it starts by computing a minimal set of generators of the lead module $\operatorname{L}(\operatorname{Syz}(\mathcal G))$ of $\operatorname{Syz}(\mathcal G) \subset \mathcal N^r$ with respect to the induced Schreyer ordering $\succ_\mathcal G$ on $\mathcal N^r$, see Algorithm~\ref{alg:leadsyz}. We then `lift' each generator of $\operatorname{L}(\operatorname{Syz}(\mathcal G))$ obtained from Algorithm \ref{alg:leadsyz} to get the generators of $\operatorname{Syz}(\mathcal G)$. This is done by computing the tail of each generator term-by-term using  \cite[Def. 4.1]{erocal2016refined}. We denote by $\phi$ the presentation map $\phi: \mathcal N^r \rightarrow \mathcal N^m$ sending the canonical bases $e_i$ to $f_i$.
\begin{definition}\label{def:lift}
    The \textit{lift} of a term $s \in$ $\operatorname{L}(\operatorname{Syz}(\mathcal G)) \subset \mathcal N^r$ with respect to $\mathcal G$ and $\succ_\mathcal G$ is a vector $\Bar s \in \mathcal N^r$ such that  $\operatorname{LT}_{\succ_\mathcal G}(\Bar s) = s$ and $\Bar s \; \in \; \operatorname{Syz}(\mathcal G).$
\end{definition}

\begin{algorithm}
\caption{LeadSyz}\label{alg:leadsyz}
\textbf{Input:} An ordered Gr\"obner basis $\mathcal G = ( f_1, \dots f_r )$ of $M$ with respect to a module monomial ordering $>_m$\\
\textbf{Output:} A minimal monomial generating set of $\operatorname{L}(\operatorname{Syz}(\mathcal G)) \subset F_1 := \mathcal N^r$ with respect to the Schreyer ordering $\succ_{\mathcal G}$.
\begin{algorithmic}[1]
\State $S := \emptyset$
\For {$1 \leq i < j \leq r$}
    \State $t:= \operatorname{LM}(m_{ji})e_i \in F_1$
    \For {$s \in S$}
    \If{ s $|$ t}
        \State $t := 0$
        \State \textbf{break}
    \ElsIf {t $|$ s}
        \State $S := S \; \backslash \; \{ s\}$
    \EndIf
    \EndFor
    \If{ $t \neq 0$}
        \State $S : = S \bigcup \{ t\}$
    \EndIf
\EndFor 
\State \textbf{return} $S$
\end{algorithmic}
\end{algorithm}

\subsection{Lifting the Lead Syzygy Module}\label{subsection:lift}

The tails of the S-vectors are computed using the algorithm \texttt{LiftSyz}, which calls the subroutines \texttt{LiftTree} and \texttt{LiftSubtree}. The details of the commutative versions of the algorithms are \cite[Alg. 3, 6, 6a]{erocal2016refined}. In this paper, we focus on the adaptation of this algorithm over $\mathcal N$ and also its parallelisability. 
\texttt{LiftTree} is applied independently to each generator $s  \in  L(\operatorname{Syz}(\mathcal G))$.

Er\"ocal, Motsak, Schreyer and Steenpass have refined this algorithm by getting rid of terms that are not divisible by any $ f_i \; \in \; \mathcal G$. These terms can therefore be left out as they produce no coefficient in the lifted syzygy. Such terms are defined as follows \cite[Def. 4.2, 4.3]{erocal2016refined}.
\begin{definition}
    Let $S \subset M$ be a set of vectors. We say that a term $t \in M$ is a \textit{lower order term with respect to} $S$ if 
    $$\operatorname{LM}(f) \nmid \operatorname{LM}(t) \; \; \text{for all } f \; \in \; S \backslash \{0\}.$$ 
\end{definition}
We denote the sum of the terms in a polynomial $g$ that are of lower order with respect to $S$ as $\operatorname{LOT}(g|S)$. 
The refined Schreyer Algorithm performs the term-by-term reduction only to the terms of $$g \; := \; \phi(s) - \operatorname{LOT}(\phi(s)|\mathcal G).$$

\begin{definition}
    For an arbitrary term $q \in \mathcal N^p$, its \textit{subtree lifting} with respect to $\mathcal G$ and $>_m$ is a vector $\bar q \in \mathcal N^p$ such that $\operatorname{LT}_{\succ_{\mathcal G}}(\bar q) = q$ and $\phi(\bar q) - \operatorname{LT}(\phi(\bar q))$ is a sum of lower order terms with respect to $\mathcal G$ and $>_m$.
\end{definition}

In the following algorithm, $m \operatorname{LT}(f_i) = t$ means equality of terms including sign coefficient, after the monomials are written in standard form, hence accounting for the anti-commutative products.

\begin{algorithm}
\caption{LiftTree}\label{alg:lifttree}
\textbf{Input:} An ordered Gr\"obner basis $\mathcal G = (f_1, \dots f_r )$ of $M$ and a term $s \in \;  \operatorname{L}(\operatorname{Syz}(\mathcal G))$. \\
\textbf{Output:} A lifting $\Bar{s} \in \operatorname{Syz}(\mathcal G)$ with respect to $\mathcal G$ and the Schreyer ordering $\succ_\mathcal G$.
\begin{algorithmic}[1]
\State $g := \phi(s)$
\State $T := \text{set of terms in }(g - \operatorname{LOT}(g | \mathcal G))$
\State $\Bar{s} := s$
\For{ all $t \in T$}
    \State choose a term $-m e_i$ such that $m \operatorname{LT}(f_i) = t$ and $s \succ_\mathcal G me_i$
    \State $\Bar{s} := \Bar{s} + \texttt{LiftSubtree}(-m e_i)$
\EndFor
\State \textbf{return} $\Bar s$
\end{algorithmic}
\end{algorithm}

\begin{algorithm}
\caption{LiftSubtree}\label{alg:liftsubtree}
\textbf{Input:} An ordered Gr\"obner basis $\mathcal G = (f_1, \dots f_r )$ of $M$ and a term $s \in \; \mathcal N^r$. \\
\textbf{Output:} A subtree lifting $\Bar{s}$ of $s$ with respect to $\succ_\mathcal G$.
\begin{algorithmic}[1]
\State $g := \phi(s) - \operatorname{LT}(\phi(s))$
\State $T := \text{set of terms in }(g - \operatorname{LOT}(g | \mathcal G))$
\State $\Bar{s} := s$
\For{ all $t \in T$}
    \State choose a term $-m e_i$ such that $m \operatorname{LT}(f_i) = t$
    \State $\Bar{s} := \Bar{s} + \texttt{LiftSubtree}(-m e_i)$
\EndFor
\State \textbf{return }$\Bar s$
\end{algorithmic}
\end{algorithm}

Algorithms \ref{alg:lifttree} and \ref{alg:liftsubtree} give rise to a workflow graph that is a tree, see Figure \ref{treegraph}. The root of the tree is the term $s \; \in \; L(\operatorname{Syz}(\mathcal G))$ that is being lifted. Each term $t = m e_i$, where $f_i \; \in \; \mathcal G$, such that $\operatorname{LM}(m f_i)  $ is a term in $ \phi(s) - \operatorname{LOT}(\phi(s)|\mathcal G)$ is then subtree lifted. This gives us only tail terms of lower order with respect to $\mathcal G$ and $>_m$.

\begin{propn}\label{prop:proof-of-correctness}
    Algorithms \ref{alg:lifttree} and \ref{alg:liftsubtree} terminate and produce liftings and subtree liftings of terms, respectively. Consequently, Algorithm \ref{alg:liftsyz} returns a Gr\"obner basis of $\operatorname{Syz}(\mathcal{G})$.
\end{propn}
\begin{proof}

The proof of correctness follows from the commutative proof in \cite{erocal2016refined} as well as the fact that when written in standard form, the terms $m\operatorname{LT}(f_i)$ and $t$ are equal, including the sign coefficient from the anti-commutative product. This ensures that the lead term is canceled at every vertex in the tree. 
\end{proof}

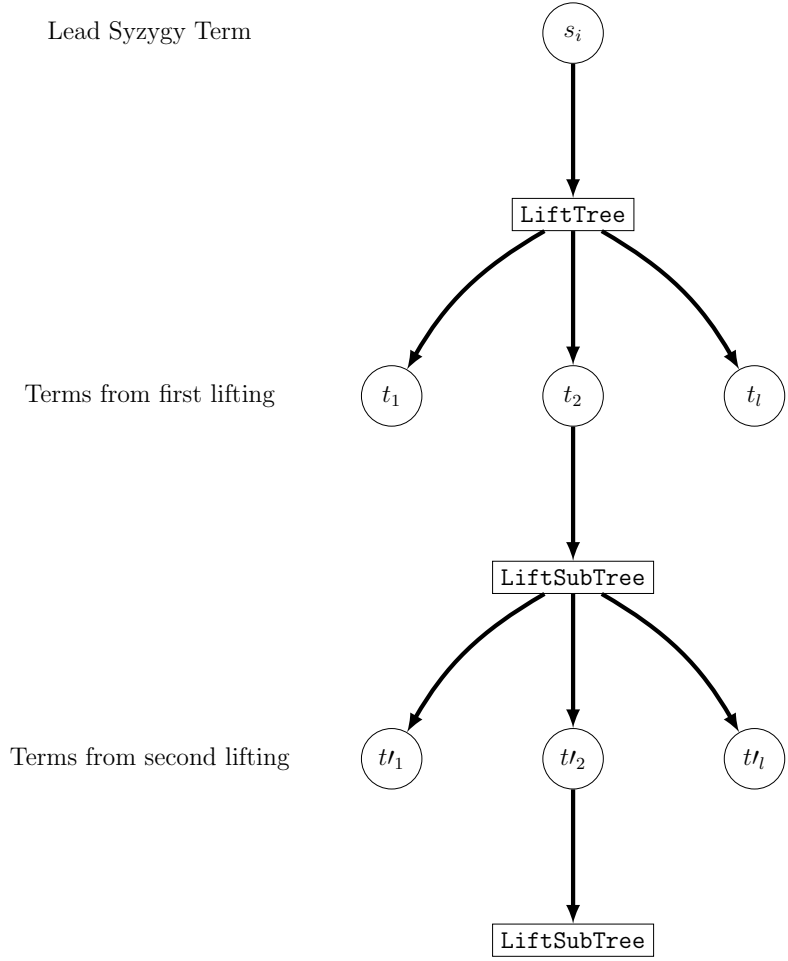
\begin{figure}[ht]
    \hspace{1cm}
     \scalebox{0.80}{
        \begin{tikzpicture}[
			   place/.style = {draw, circle},
            trans/.style = {draw, rectangle},
			]
	
		\node[place,minimum size=1cm] (si) at (0, 9) {$s_i$};
        \node[place,minimum size=1cm] (t1) at (-3, 3) {$t_1$};
        \node[place,minimum size=1cm] (t2) at (0, 3) {$t_2$};
        \node[place,minimum size=1cm] (tl) at (3, 3) {$t_l$};
        \node[place,minimum size=1cm] (tp1) at (-3, -3) {$t\prime _1$};
        \node[place,minimum size=1cm] (tp2) at (0, -3) {$t\prime _2$};
        \node[place,minimum size=1cm] (tpl) at (3, -3) {$t\prime _l$};
         \node[trans] (LiftTree)   at (0, 6) {\texttt{LiftTree}};
         \node[trans] (LiftSubTree1)   at (0, 0) {\texttt{LiftSubTree}};
         \node[trans] (LiftSubTree2)   at (0, -6) {\texttt{LiftSubTree}};

        \node (leadsyz) at (-7, 9) {\text{Lead Syzygy Term}};
        \node (terms1) at (-7, 3) {\text{Terms from first lifting}};
        \node (terms2) at (-7, -3) {\text{Terms from second lifting}};

		\draw[->, >=latex, line width=2pt]
		(si)
		to [out=270, in=90, looseness=1.0]
		(LiftTree); 
        
        \draw[->, >=latex, line width=2pt]
		(LiftTree)
		to [out=270, in=90, looseness=1.0]
		(t2); 
        \draw[->, >=latex, line width=2pt]
		(LiftTree)
		to [out=210, in=60, looseness=1.0]
		(t1);
        \draw[->, >=latex, line width=2pt]
		(LiftTree)
		to [out=330, in=120, looseness=1.0]
		(tl);

        \draw[->, >=latex, line width=2pt]
		(t2)
		to [out=270, in=90, looseness=1.0]
		(LiftSubTree1);
  
        \draw[->, >=latex, line width=2pt]
		(LiftSubTree1)
		to [out=270, in=90, looseness=1.0]
		(tp2);
        \draw[->, >=latex, line width=2pt]
		(LiftSubTree1)
		to [out=210, in=60, looseness=1.0]
		(tp1);
        \draw[->, >=latex, line width=2pt]
		(LiftSubTree1)
		to [out=330, in=120, looseness=1.0]
		(tpl);

        \draw[->, >=latex, line width=2pt]
		(tp2)
		to [out=270, in=90, looseness=1.0]
		(LiftSubTree2);
		\end{tikzpicture}}
    \caption{Workflow graph of \texttt{LiftTree} and \texttt{LiftSubtree}}\label{treegraph}
\end{figure}

Algorithm \ref{alg:lifttree} terminates once there are no terms in $T$ on every branch of the tree. Once the lead terms $\{ s_i \}  \in  \operatorname{L}(\operatorname{Syz}(\mathcal G))$ are lifted, they are then added to finally give us the required Gr\"obner basis of $\operatorname{Syz}(\mathcal G)$.

\begin{algorithm}
\caption{LiftSyz}\label{alg:liftsyz}
\textbf{Input:} An ordered Gr\"obner basis $\mathcal G = ( f_1, \dots f_r )$ of $M$, with respect to a module monomial ordering $>_m$\\
\textbf{Output:} A Gr\"obner basis of $\operatorname{Syz}( \mathcal G) \subset \mathcal N^r$ with respect to the Schreyer ordering $\succ_\mathcal G$.
\begin{algorithmic}[1]
\State $S = \texttt{LeadSyz}(G)$
\State $\Bar{S} = \emptyset$
\For{$s \in S$}
    \State $\Bar{s} = \texttt{LiftTree}(s)$
    \State $\Bar{S} = \Bar{S} \bigcup\{ \Bar{s} \}$
\EndFor
\State \textbf{return} $\Bar{S}$
\end{algorithmic}
\end{algorithm}

\begin{example}\label{eg:lift-by-hand}
    Let $\mathcal N$ be equipped with degree reverse lexicographic monomial ordering, and $M \subset F_0 := \mathcal N^6$ a module, which is the lifting of the $E$-module $\Bar M$ generated by the Gr\"obner basis $\mathcal G$ with ordered elements
    \begin{gather*}
        f_1 = x e_1 + y e_2 + z e_4,\\
        f_2 = x e_2 + y e_3 + z e_5,\\
        f_3 = x e_4 + y e_5 + z e_6,\\
        f_4 = xy e_3 + xz e_5,\\
        f_5 = xy e_5 + xz e_6,\\
        f_6 = xyz e_6.
    \end{gather*}
    Thus $M$ is generated by a Gr\"obner basis of cardinality $24$, where $f_7 $ to $ f_{24}$ are
    \begin{center}
\begin{tabular}{ c c c }
 $f_7 = x^2 e_1$ & $\dots$ & $f_{12} = x^2 e_6$ \\ 
 $f_{13} = y^2 e_1$ & $\dots$ & $f_{18} = y^2 e_6$ \\  
 $f_{19} = z^2 e_1$ & $\dots$ & $f_{24} = z^2 e_6$     
\end{tabular}
\end{center}
     Here, $\{e_1, \dots, e_6\}$ is a canonical basis of $\mathcal N^6$. We know that $\operatorname{Syz}(\mathcal G) \subset \mathcal N^{24}$, whose canonical basis we denote by $\{e_1^\prime , \dots , e_{24}^\prime\}$. Let us assume that the monomial ordering on $\mathcal N^6$ is compatible with that of $\mathcal N$, and gives preference to monomials over components. 

To compute the first lead syzygy term, we notice that there are no
syzygies between $f_1,\dots,f_6$. Thus, the first term computed by our
algorithm in $\operatorname{L}(\operatorname{Syz}(M))$ is
$m_{7,1}e_1^\prime$, where
\begin{gather*}
    m_{7,1}e_1^\prime
    =
    \operatorname{lquot}\bigl(
        \operatorname{LCM}(\operatorname{LM}(f_7),
                           \operatorname{LM}(f_1)),
        \operatorname{LT}(f_1)
    \bigr)e_1^\prime
    =
    xe_1^\prime.
\end{gather*}
Indeed, the two terms $xe_1^\prime$ and $e_7^\prime$ have the same
image lead monomial $x^2e_1$, and $xe_1^\prime\succ_{\mathcal G}
e_7^\prime$ since $1<7$.
We therefore apply \texttt{LiftTree} to $xe_1^\prime$. We have
\begin{align*}
    \phi(xe_1^\prime)
    =
    x^2e_1+xye_2+xze_4,
    \qquad
    T=\{x^2e_1,xye_2,xze_4\}.
\end{align*}
Since
\begin{align*}
    x^2e_1=\operatorname{LT}(f_7),\qquad
    xye_2=(-y)\operatorname{LT}(f_2),\qquad
    xze_4=(-z)\operatorname{LT}(f_3),
\end{align*}
we apply \texttt{LiftSubtree} to
$-e_7^\prime$, $ye_2^\prime$, and $ze_3^\prime$, respectively.
The first of these branches terminates immediately. For the second
branch we obtain
\begin{align*}
    g
    :=
    \phi(ye_2^\prime)
    -
    \operatorname{LT}(\phi(ye_2^\prime))
    =
    y^2e_3+yze_5,
    \qquad
    T=\{y^2e_3\},
\end{align*}
since $yze_5$ is a lower order term. Hence we add $-e_{15}^\prime$.
For the third branch we obtain
\begin{align*}
    g
    :=
    \phi(ze_3^\prime)
    -
    \operatorname{LT}(\phi(ze_3^\prime))
    =
    -yze_5+z^2e_6,
    \qquad
    T=\{z^2e_6\},
\end{align*}
since $-yze_5$ is a lower order term. Hence we add $-e_{24}^\prime$.
All three branches then terminate.
The workflow graph of the complete term-by-term lifting can be seen
in Figure~\ref{figure:lifting}.
\begin{figure}[ht]
\centering
\scalebox{0.72}{
\begin{tikzpicture}[
    place/.style={draw,circle},
    trans/.style={draw,rectangle},
]
    \node[place,minimum size=1cm] (si) at (0,11)
        {$xe_1^\prime$};

    \node[trans] (LiftTree) at (0,9) {
        \begin{tabular}{@{}c@{}}
            $g=x^2e_1+xye_2+xze_4$\\
            $T=\{x^2e_1,xye_2,xze_4\}$
        \end{tabular}
    };

    \node[place,minimum size=1cm] (q0) at (-6,6)
        {$-e_7^\prime$};
    \node[place,minimum size=1cm] (q1) at (0,6)
        {$ye_2^\prime$};
    \node[place,minimum size=1cm] (q2) at (6,6)
        {$ze_3^\prime$};

    \node[trans] (LiftSubTree0) at (-6,4) {$g=0$};

    \node[trans] (LiftSubTree1) at (0,4) {
        \begin{tabular}{@{}c@{}}
            $g=y^2e_3+yze_5$\\
            $T=\{y^2e_3\}$
        \end{tabular}
    };

    \node[trans] (LiftSubTree2) at (6,4) {
        \begin{tabular}{@{}c@{}}
            $g=-yze_5+z^2e_6$\\
            $T=\{z^2e_6\}$
        \end{tabular}
    };

    \node[place,minimum size=1cm] (r1) at (0,1)
        {$-e_{15}^\prime$};
    \node[place,minimum size=1cm] (r2) at (6,1)
        {$-e_{24}^\prime$};

    \node[trans] (LiftSubTree3) at (0,-1) {$g=0$};
    \node[trans] (LiftSubTree4) at (6,-1) {$g=0$};

    \draw[->,>=latex,line width=2pt] (si) -- (LiftTree);

    \draw[->,>=latex,line width=2pt]
        (LiftTree) to[out=210,in=90] (q0);
    \draw[->,>=latex,line width=2pt]
        (LiftTree) -- (q1);
    \draw[->,>=latex,line width=2pt]
        (LiftTree) to[out=330,in=90] (q2);

    \draw[->,>=latex,line width=2pt]
        (q0) -- (LiftSubTree0);
    \draw[->,>=latex,line width=2pt]
        (q1) -- (LiftSubTree1);
    \draw[->,>=latex,line width=2pt]
        (q2) -- (LiftSubTree2);

    \draw[->,>=latex,line width=2pt]
        (LiftSubTree1) -- (r1);
    \draw[->,>=latex,line width=2pt]
        (LiftSubTree2) -- (r2);

    \draw[->,>=latex,line width=2pt]
        (r1) -- (LiftSubTree3);
    \draw[->,>=latex,line width=2pt]
        (r2) -- (LiftSubTree4);
\end{tikzpicture}}
\caption{\texttt{LiftTree} applied to $xe_1^\prime$}
\label{figure:lifting}
\end{figure}
We can infer the total syzygy vector from the terms in
Figure~\ref{figure:lifting}. In this example, it is
\[
    xe_1^\prime-e_7^\prime
    +ye_2^\prime+ze_3^\prime
    -e_{15}^\prime-e_{24}^\prime.
\]
From Theorem~\ref{thm:nsyztoesyz}, its image
\[
    xe_1^\prime+ye_2^\prime+ze_3^\prime
\]
is an element of the Gr\"obner basis obtained from the refined
Schreyer algorithm.
\end{example}

\subsection{Avoiding Evaluations of the Schreyer Ordering}

In order to use the Schreyer ordering without an explicit monomial ordering on the syzygy modules, we use Algorithm~\ref{alg:schrey-order} that checks if $v_1 \succ_{\mathcal G} v_2$, where $v_1, v_2$ are terms in a module in the Schreyer resolution of $M$. This is done using a list $L$ that keeps track of the terms of the generators in the Schreyer frame that is thus far computed. Let $M$ be the module whose free resolution we are computing. Let $F^i$ denote the $i^{th}$ module in the resolution. Then, in order to compare elements in $F^{i+1}$, we use the list of morphisms $\xi_i ,\dots, \xi_1 $ of the Schreyer frame as in Definition~\ref{def:schreyer-frame} and the ordering $>_m$.

\begin{algorithm}[ht]
\caption{SchreyerOrderCheck}\label{alg:schrey-order}
\textbf{Input:} Terms $v_1$ and $v_2$, an ordered list of ordered lists, $(\xi_i, \dots, \xi_1)$, that contains the lead terms of the generators of each module in the resolution that is computed so far. The first entry of the list is the first module in the resolution. $>_m$, the module ordering on $M$.\\
\textbf{Output:} Boolean value \texttt{true} if $v_1 > v_2$, \texttt{false} if $v_1 < v_2$.
\begin{algorithmic}[1]
\State Write $\operatorname{LT}(v_1) := m_ke_k$, $\operatorname{LT}(v_2) := m_je_j$
\State $w_1 = m_k \cdot \xi_i(e_k)$
\State $w_2 = m_j\cdot \xi_i(e_j)$
\If {$i == 1$}
    \State \textbf{return} $w_1 >_m w_2$
\Else
    \State \textbf{return} \texttt{SchreyerOrderCheck}($w_1, w_2, (\xi_{i-1}, \dots, \xi_1 ))$
\EndIf
\end{algorithmic}
\end{algorithm}

Note that Hilbert's syzygy theorem does not necessarily hold over the exterior algebra. This means that a free resolution of some $E$ module may be infinite, as opposed to modules over the polynomial ring, whose free resolutions always have a finite length. Hence, when computing free resolutions over the exterior algebra, one must specify the homological degree until which the desired resolution should be computed.

\section{Sheaf Cohomology via BGG Correspondence}
\label{bgg}
This paper is motivated by an important application of free resolutions over the exterior algebra, namely, computing sheaf cohomology numbers of coherent sheaves over projective space. This is done using an example of a Koszul duality between the derived category of bounded complexes of coherent sheaves over projective space and the stable category of finitely generated modules over the exterior algebra - called the \textit{Bernstein-Gel'fand-Gel'fand correspondence} \cite{bgg78}. An explicit duality was constructed by Eisenbud, Fløystad and Schreyer in \cite{eisenbud2003sheaf}, showing that taking the homology of a complex of sheaves corresponds to the linear part of its Koszul dual - a resolution over the exterior algebra, called the `Tate Resolution', gives an isomorphism, allowing us to obtain the cohomology numbers.

Let $V$ be an $n-$dimensional vector space over an algebraically closed field $K$ and $W$ be its dual. We mean by $\{e_1, \dots, e_n\}$ a basis of $V$, and by $\{ x_1, \dots , x_n \}$ the dual basis of $W$. We give each $e_i$ degree $-1$,  and each $\{x_i\}$ degree $1$. We denote by $S$ the symmetric algebra over the vector space $W$ and by $E$ the exterior algebra $\oplus_{i=0}^n\bigwedge^i V$. 

Finitely generated modules over the symmetric algebra sheafify to coherent sheaves over the projective space $\textbf{P}(W)$. Let $M$ be a graded $S$-module. Then we write $M$ as $$M = \oplus M_d,$$where each $M_d$ is a finite dimensional $K$-vector space, such that there exist multiplication maps $$\mu_d: W \otimes M_d \rightarrow M_{d+1}.$$ 
From $\mu_d$ we get a linear matrix that defines a map from $E \otimes M_d \rightarrow E(-1) \otimes M_{d+1}$.
We construct a functor $\textbf{R}$ that associates to every graded $S-$module a linear free complex of $E-$modules.

\begin{definition}\label{definition:R}
    Denoting $\sigma$ as the map given by $\sum_i e_i \otimes x_i$, given an $S$-module $M$, \textbf{R}$(M)$ is defined as
$$\textbf{R}(M): \dots \xrightarrow[]{\sigma} E \otimes_K M_d \xrightarrow[]{\sigma} E(-1) \otimes_K M_{d+1} \xrightarrow[]{\sigma} \dots.$$
\end{definition}

One can check that $\textbf{R}(M)$ is a complex of modules over $E$. Note that since $\operatorname{Hom}(E,K) \cong E(-n)$, we may rewrite the same functor as 
\begin{align}\label{equation:R}
    \begin{split}
        \textbf{R}(M): \dots \xrightarrow[]{\sigma} \operatorname{Hom}_K(E,M_d )  \xrightarrow[]{\sigma} \operatorname{Hom}_K(E,M_{d+1}) \xrightarrow[]{\sigma} \dots,
    \end{split}
\end{align}
where $\sigma(\alpha) = [e\mapsto \sum_i x_i \alpha(e_i e)]$.
This map is often referred to as the BGG map, and is implemented as \texttt{bgg} in \textsc{Macaulay2} and as \texttt{sym\text{Ext}} in \textsc{Singular}. \textbf{R} gives an equivalence between graded $S$-modules and linear free $E-$complexes \cite[Prop. 2.1]{eisenbud2003sheaf}. \textbf{R} is eventually exact, as shown in \cite[Thm. 3.1]{eisenbud2003sheaf}. The `tail' $\textbf{R}(M_{\geq d})$ for any $d>\operatorname{reg}(M)$ is exact at all positions $i>d$. 

We construct a doubly infinite exact sequence called the Tate resolution by adjoining the minimal free resolution of $\operatorname{ker}(\sigma_d) =\operatorname{ker}[\operatorname{Hom}_K(E,M_d)\rightarrow \operatorname{Hom}_K(E,M_{d+1})]$ to the tail \textbf{R}$(M_{\geq d})$:
\begin{equation}\label{tate-resolution}
    \begin{tikzcd}[row sep=0.7em, column sep=0.25em]
    \textbf{T}: & \dots \arrow[r] & T^{d-2}(M) \arrow[rr] && T^{d-1}(M) \arrow[rr] \arrow[dr] && \text{Hom}_K(E,M_d) \arrow[rr] && \text{Hom}_K(E,M_{d+1}) \arrow[rr] && \dots \\
    &&&&& \operatorname{ker}(\sigma_d) \arrow[ur] \arrow[dr]\\
    &&&& 0 \arrow[ur] && 0
    \end{tikzcd}
\end{equation}
The first row of \eqref{tate-resolution} is the Tate resolution, which does not depend on $d$, only on ensuring that a module $\operatorname{Hom}_K(E, M_d)$ from the exact tail of $\textbf{R}$ is chosen. Moreover, given two modules $M$ and $N$ such that $M_{\geq r} \cong N_{\geq r}$ for some integer $r$ implies that the coherent sheaves $\Tilde{M} \cong \Tilde{N}$ (by Serre's Theorem \cite{serre1955faisceaux}). By \cite[Section 4]{eisenbud2003sheaf}, the Tate resolution does not depend on the module $M$, only on the sheaf $\mathcal F$ which it represents. We may interchange $M$ and $\mathcal F$ when talking of the Tate resolution. The following result, \cite[Thm. 4.1]{eisenbud2003sheaf}, gives us the isomorphism between the modules of the Tate resolution and the cohomology modules of $\tilde M$.

\begin{theorem}[\cite{eisenbud2003sheaf}]\label{theorem:main-tate}
    Using the notation $\omega_E = E \otimes \bigwedge^n W$, we can write the $i^{th}$ term of the Tate resolution $\textbf{T}(M) = \textbf{T}(\mathcal F)$ with cohomological degree $i$ as 
    \begin{gather*}
        T^i = \oplus_j H^j (\mathcal F(i-j))\otimes \omega_E(j-i)
    \end{gather*}
\end{theorem}

We use Theorem \ref{theorem:main-tate} to compute useful information about the cohomology of the given sheaf. The following result is a reformulation of Theorem \ref{theorem:main-tate} in \cite{stillman2003computing}.

\begin{corollary}\label{cor:final}
    Let $M$ be a graded $S$ module that represents a sheaf $\mathcal F$. Then the Betti diagram of the Tate resolution $\textbf{T}(M)$ has the following form:
    \begin{equation}
        \begin{tikzcd}[row sep=0.8em, column sep=0.8em]
            \dots & h^0(\mathcal F(1)) & h^0(\mathcal F) & h^0(\mathcal F(-1)) & \dots \\
            \dots & h^1(\mathcal F) & h^1(\mathcal F(-1)) & h^1(\mathcal F(-2)) & \dots \\
            & \vdots && \vdots \\
            \dots & h^{n-1}(\mathcal F(-n +2)) & h^{n-1}(\mathcal F(-n+1)) & h^{n-1}(\mathcal F(-n)) & \dots 
        \end{tikzcd}
    \end{equation}
\end{corollary}

This method of computing cohomology tables is much preferred to that using \v{C}ech complexes due to computation time. This is why the functions \texttt{cohomologyTable} in Macaulay2, \texttt{sheafCohBGG} in \textsc{Singular} and \texttt{sheaf\textunderscore cohomology} in Oscar all use the BGG correspondence. 

The critical step in computing the Tate resolution, in order to obtain the Betti table in Corollary \ref{cor:final}, is to compute a minimal free resolution over the exterior algebra. We consider a resolution to be minimal if the homogeneous entries of the differentials lie in the augmentation ideal $\langle x_1, \dots, x_n\rangle$. Since $E$ has no homogeneous units other than non-zero constants in $K$ and $M$ is graded, a graded minimal resolution has no non-zero constant entries in the differentials. To obtain such a minimised free resolution, we can employ our adapted version of the refined Schreyer algorithm, and then minimise the obtained resolution using the \textsc{Singular} command \texttt{minres}, that eliminates constants using linear algebra. 

\begin{example}\label{eg:minimal-res}
    Using Algorithm \ref{alg:schrey-order}, we can compute a Schreyer resolution over the exterior algebra using \textsc{Singular}. We compute a Schreyer resolution of the module $\Bar M$ defined in Example \ref{eg:lift-by-hand}.

    \begin{lstlisting}
    > sresExt(M,4);
     6      6      10      15      21      
    E <--  E <--  E <--   E <--   E

    0      1      2       3       4       
    resolution not minimized yet
    \end{lstlisting}

    \vspace{-0.2cm}
    We can also compute minimal free resolutions over $E$ in \textsc{Singular}, which is minimised after the entire resolution is computed. 
    We see that the Betti table of \texttt{sresExt} in this range will coincide with the \textsc{Singular} command \texttt{mres}, except for the last column on the right.
    
    \begin{lstlisting}
    > resolution res_ = sresExt(M,4);
    > betti(res_);
    6,3,1,0,0,
    0,0,0,0,3,
    0,0,0,1,10 
    > mresExt(M,4);
     6      3      1      1      3      
    E <--  E <--  E <--  E <--  E

    0      1      2      3      4      
    \end{lstlisting}
    
\end{example}

\section{Notes on Parallelisation}\label{section:parallel}

\subsection{Overview}

In this section, we discuss the potential of a massively parallel formulation and implementation of the refined Schreyer algorithm. This would allow one to handle large examples of free resolutions over the exterior algebra as they occur, for example, in the context of the BGG correspondence. The main observation of this section is that the framework of \cite{gnawali2024massively} for free resolutions over polynomial rings in the commutative case can be extended to the exterior algebra case. This is the subject of ongoing joint work with Gnawali. 

The source of parallelism which we exploit in the refined Schreyer algorithm is that each term to be lifted can be handled independently of the others. The dependency between computations is encoded in a graph. We employ a breadth-first-search to traverse the graph. 

\subsection{The \textsc{Singular}/\textsc{GPI-Space} Framework}

Starting with \cite{BDPRR21}, a framework for massively parallel computations in computer algebra has been developed by the first author and coauthors. This framework has already proven its abilities in various applications ranging from algebraic geometry to high-energy physics, see for example \cite{modular, pfd}. The framework relies on the workflow management system \textsc{GPI-Space} to implement a coordination layer in the sense of Gelernter \cite{gelernter}, while \textsc{Singular} is employed for the computation layer and for serialization. 

To realize a coordination language, the formalism of Petri nets is used \cite{carl-petri}. A \textbf{Petri net} is a directed bipartite graph with two types of vertices: \textbf{places}, which can hold structureless tokens, and \textbf{transitions} which are the computational units. A transition is \textbf{enabled} if all its input places hold at least one token. An enabled transition can \textbf{fire} consuming one token from each input place and producing one token on each output place. A Petri net is executed by firing random enabled transitions. For practical purposes we employ an extended language of Petri nets where tokens are pieces of data (or references to those, implementing reference counting), transitions can impose conditions on the accepted collections of tokens, transition can produce multiple tokens on an output place (shown as double arrows), input arrows to transitions can be read only (dashed arrows), and it is possible to ask for the number of tokens on a place.
Moreover, we take into account that transitions require time to execute. Parallelism arises from firing distinct transitions at overlapping times, called \textbf{task parallelism}, as well as firing the same transition multiple times if the input places each hold a sufficient amount of tokens, called \textbf{data parallelism}. The consistency of the coordination program is not impacted by either.
\begin{example} Figure \ref{petrinet} depicts a Petri net where transitions $T_1$ and $T_2$ can fire in parallel. Moreover, transition $T_2$ can fire in a data-parallel fashion.
\end{example}

\begin{figure}[htbp]
\centering
\begin{tikzpicture}[node distance=1.8cm, on grid, >=stealth', auto]

  \node[place,tokens=1] (p1) {};
  \node[place,tokens=5] (p2) [below=of p1] {};
  \node[place,tokens=4] (p3) [below=of p2] {};

  \node[transition] (t1) [right=of p1, yshift=-0.9cm] {$T_1$};
  \node[transition] (t2) [right=of p3, yshift=0.9cm] {$T_2$};

  \node[place,tokens=0] (p4) [right=of t1, yshift=-0.9cm] {};

  \draw[->] (p1) -- (t1);
  \draw[->] (p2) -- (t2);
  \draw[->] (p3) -- (t2);
  \draw[->] (t1) -- (p4);
  \draw[->] (t2) -- (p4);

  \begin{scope}[xshift=6cm]
    \node[place,tokens=0] (p1') {};
    \node[place,tokens=1] (p2') [below=of p1'] {};
    \node[place,tokens=0] (p3') [below=of p2'] {};

    \node[transition] (t1') [right=of p1', yshift=-0.9cm] {$T_1$};
    \node[transition] (t2') [right=of p3', yshift=0.9cm] {$T_2$};

    \node[place,tokens=5] (p4') [right=of t1', yshift=-0.9cm] {};

    \draw[->] (p1') -- (t1');
    \draw[->] (p2') -- (t2');
    \draw[->] (p3') -- (t2');
    \draw[->] (t1') -- (p4');
    \draw[->] (t2') -- (p4');
  \end{scope}

\end{tikzpicture}
\caption{Petri net before (left) and after (right) firing $T_1$ and $T_2$ on $\geq 5$ cores in parallel.}
\label{petrinet}
\end{figure}
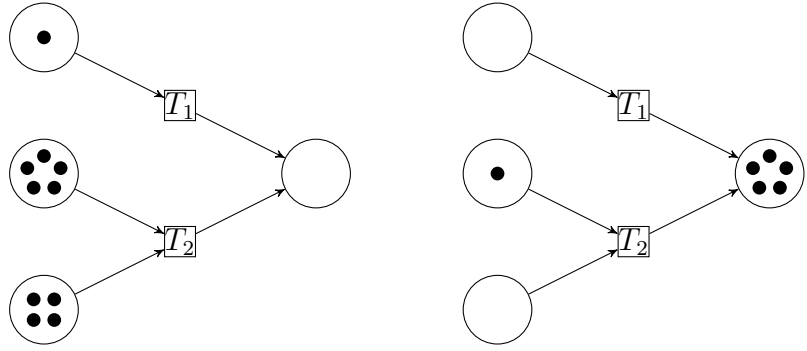

The approach of Gnawali from \cite{gnawali2024massively} can be directly carried over to the exterior algebra setting by considering the algorithms \texttt{LiftTree} and \texttt{LiftSubtree} of Section \ref{section:schreyer-algorithm}. As those algorithms are recursive in nature, Gnawali first translates them (in their polynomial ring variants) into an iterative form and then develops a Petri net for coordinating the syzygy and thus the resolution algorithm. 

 Recall from Section~\ref{subsection:lift} that the lifting of lead terms is done over the ring $\mathcal N$. The parallelised syzygy algorithm is also performed over $\mathcal N$. The resolution over $E$ is obtained using relative Gr\"obner bases, as described in Section~\ref{section:syz-ext}. 
Note that in our setting, resolutions over $E$ will not necessarily terminate, and thus we have to specify a bound on the homological degree to be computed.

The lifting summands are combined by the parallel addition subnet in Figure \ref{fig:Parallel_sum}.

 \begin{figure}[ht]
 \centering
 	\scalebox{0.80}{
 		\begin{tikzpicture}[
 			   place/.style = {draw, circle},
            trans/.style = {draw, rectangle},
			]
	
		\node[place] (I) at (-6, 0) {\texttt{Input}};
          \node[place] (L1) at (-3, 0) {\texttt{ctrl}};
            \node[place] (R1) at (0, 0) {\texttt{rhs}};
       \node[place] (L) at (2, 2) {\texttt{L}};
        \node[place] (R) at (2, -2) {\texttt{R}};

         \node[trans] (T1)   at (-2, 2) {\texttt{LHS}};
	    \node[trans] (T2)   at (-2, -2) {\texttt{RHS}};
		
		\node[trans] (O) at ( 4, 0) {\texttt{Add}};

		\draw[->, >=latex, line width=2pt]
		(I)
		to [out=30, in=210, looseness=1.0]
		(T1); 
		
		\draw[->, >=latex, line width=2pt]
		(I)
		to [out=330, in=150, looseness=1.0]
		(T2); 
		
		\draw[->, >=latex, line width=2pt]
       (T1)
       to
       (L);
       
 \draw[->, >=latex, line width=2pt]
		(T2)
		to [out=0, in=180, looseness=1.0]
		(R);      
  \draw[->, >=latex, line width=2pt]
		(L1)
		to [out=60, in=240, looseness=1.0]
		(T1);

  \draw[->, >=latex, line width=2pt]
		(T2)
		to [out=120, in=270, looseness=1.0]
		(L1);

  \draw[->, >=latex, line width=2pt]
		(R1)
		to [out=260, in=60, looseness=1.0]
		(T2);

  \draw[->, >=latex, line width=2pt]
		(T1)
		to [out=300, in=120, looseness=1.0]
		(R1);

   \draw[->, >=latex, line width=2pt]
		(R)
		to [out=30, in=210, looseness=1.0]
		(O);
  
    \draw[->, >=latex, line width=2pt]
		(L)
		to [out=330, in=120, looseness=1.0]
		(O);

    \draw[->, >=latex, line width=2pt]
		(O)
		to [out=270, in=270, looseness=1.0]
		(I);
  
		\end{tikzpicture}}
	\caption{Petri net for parallel addition}
	\label{fig:Parallel_sum}
\end{figure}
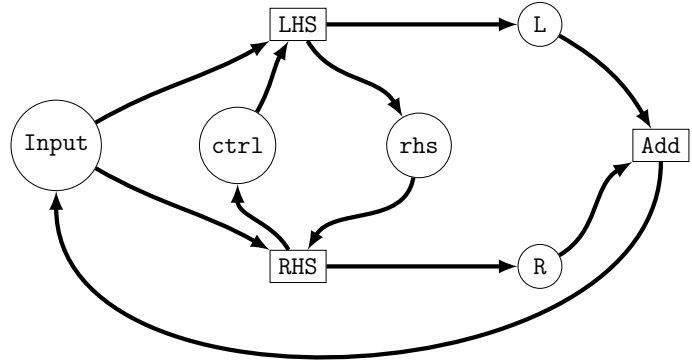

\subsection{Iterative Schreyer Algorithm}

The algorithm \texttt{LiftSubtree} of Section \ref{section:schreyer-algorithm} is called recursively and hence is not suitable for use in parallel computation. Rather, we reformulate it into an iterative algorithm.

Recall that in \texttt{LiftTree} and \texttt{LiftSubtree}, see Subsection \ref{subsection:lift}, we start with a term $s \in \operatorname{L}(\operatorname{Syz}(\mathcal G))$ where $\mathcal G=(f_1, \dots, f_r)$ is an ordered Gr\"obner basis. In the iterative version, we compute a set $T$ of all the terms in $\phi(s)$ which are not lower order terms. Then, for each $t \in T$, we compute a term $-m e_i$ such that $m \operatorname{LT}(f_i) = t$. In \texttt{LiftTree}, one extra condition, namely that $s \succ_\mathcal G m e_i$, must be satisfied. The term $-m e_i$ is then added to the tail of the syzygy and also lifted. These liftings can be done iteratively instead of recursively. Here we present an iterative version of Algorithm \texttt{LiftSubtree}, see Algorithm \ref{alg:lifttree-2}. 

\begin{algorithm}[ht]
\caption{Iterative version of LiftSubtree}\label{alg:lifttree-2}
\textbf{Input:} A Gr\"obner basis $\mathcal G = \{f_1, \dots f_r \}$ of $M$ with respect to $>_m$ and a term $s \in $L(Syz($\mathcal G))$. \\
\textbf{Output:} A lifting $\Bar{s} \in Syz(\mathcal G)$ with respect to $\mathcal G$ the Schreyer ordering $\succ_m$.
\begin{algorithmic}[1]
\State $g := \phi(s) - \text{LT}(\phi(s))$
\State $T := \text{set of terms in }(g - \text{LOT}(g | \mathcal G)$
\State $L := $ empty list
\For{ all $t \in T$}
    \State choose a term $-m e_i$ such that $m \text{LT}(f_i) = t$
    \State append $-m e_i$ to $L$
\EndFor
\State \textbf{return }$L$
\end{algorithmic}
\end{algorithm}

In this iterative version of the algorithm, each iteration generates a term $-me_i$ of the syzygy vector. These terms are collected in a list $L$, such that they must simply be added at the end. In a similar way, an iterative version of \texttt{LiftTree} can be formulated.
Since each term $me_i$ is multiplied by $-1$, the terms can simply be added after \texttt{LiftTree} and \texttt{LiftSubtree} have run through all terms, which can be done in parallel.  

\subsection{Formulation as a Petri net}

Figure \ref{fig:petri-net-syz} is a graph that encodes the Petri net of the parallel computation of free resolutions in the commutative setting in \cite{gnawali2024massively},\cite{G25}. In this net, \texttt{LeadSyz} produces lead terms of syzygies, \texttt{Lift} and \texttt{SubLift}
perform liftings or subtree liftings, and the subnet for parallel addition, as in Figure \ref{fig:Parallel_sum}, combines the resulting summands. The counter and
certificate ensure that all terms are added.

This net is also applicable for the computation of syzygies over $\mathcal N$, because module addition is commutative and all non-commutative arithmetic
is local to the transitions. In the adaptation, \texttt{LeadSyz}, \texttt{Lift}, and \texttt{SubLift} perform the
adapted algorithms of Section~\ref{section:schreyer-algorithm}. 
The resolution is computed iteratively by first finding an $\mathcal N$-Gr\"obner basis of the syzygy module of the lifted module, then applying $\rho$ from Theorem~\ref{thm:nsyztoesyz} and deleting images equal to $0$. 

The parallel implementation relies on
kernel-level implementations of the transitions. These also serve as a fast kernel-level
sequential version. Implementation details and performance results are deferred to future
work.

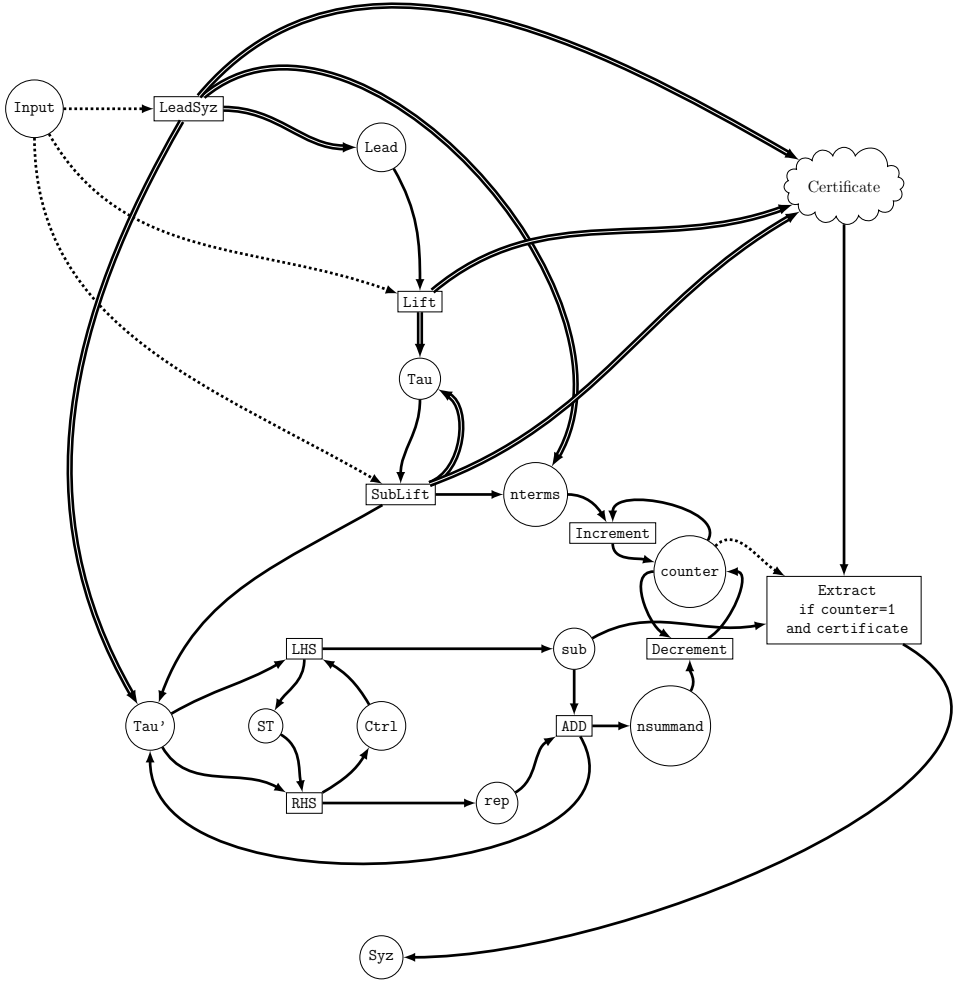
\begin{figure}[ht]

    \hspace{1cm}
    \scalebox{0.51}{
    
        \begin{tikzpicture}[
            place/.style = {draw, circle},
            trans/.style = {draw, rectangle},
        ]
\node[place] (I) at (-12, 2) {\texttt{Input}};
\node[trans] (T1) at (-8, 2 ) {\texttt{LeadSyz}};
 \node[place] (p1)   at ( -3, 1) {\texttt{Lead}};
\node[trans] (T3)    at (-2,  -3) {\texttt{Lift}};
\node[place] (p3)    at (-2,  -5) {\texttt{Tau}};
\node[trans] (T4)    at (-2.5,  -8) {\texttt{SubLift}};
\node[place] (p4)    at (2,  -12) {\texttt{sub}};
\node[trans] (T5)    at ( -5,  -12) {\texttt{LHS}};
\node[trans] (T7)    at (-5,  -16) {\texttt{RHS}};
\node[place] (p10)    at (0,  -16) {\texttt{rep}};
\node[place] (p12)    at (-9,  -14) {\texttt{Tau'}};
\node[place] (p13)    at (-3,  -14) {\texttt{Ctrl}};
\node[place] (p14)    at (-6,  -14) {\texttt{ST}};
\node[trans] (T9)    at (2,  -14) {\texttt{ADD}};
\node[place] (p15)    at (-3,  -20) {\texttt{Syz}};
\node[place] (p18)    at (1,  -8) {\texttt{nterms}};
\node[place] (p20)    at (5,  -10) {\texttt{counter}};
\node[trans] (T14)    at (5,  -12) {\texttt{Decrement}};
\node[draw, rectangle, text width=3cm, align=center, minimum width = 4cm] (T13) at (9, -11) {
        \begin{tabular}{@{}c@{}}
            \texttt{Extract}\\
            \texttt{if counter=1}\\
            \texttt{and certificate}
        \end{tabular}
    };
\node[place] (p19)    at (4.5,  -14) {\texttt{nsummand}};
\node[trans] (T12)    at (3,  -9) {\texttt{Increment}};
\node[cloud, cloud puffs=15.7, cloud ignores aspect, minimum width=2cm, minimum height=2cm, align=center, draw](cloud1) at (9, 0) {Certificate};

\draw[dotted, ->, >=latex, line width=2pt] (I) to [out=00, in=180, looseness=1.0] (T1);

   \draw[double,->, >=latex, line width=2pt] (T1) to [out=0, in=180, looseness=1.0] (p1);

   \draw[->, >=latex, line width=2pt] (p1) to [out=300, in=90, looseness=1.0] (T3);

   \draw[dotted, ->, >=latex, line width=2pt] (I) to [out=300, in=160, looseness=1.0] (T3);

   \draw[double,->, >=latex, line width=2pt] (T3) to [out=270, in=90, looseness=1.0] (p3);

   \draw[->, >=latex, line width=2pt] (p3) to [out=270, in=90, looseness=1.0] (T4);

   \draw[double, ->, >=latex, line width=2pt] (T4) to [out=20, in=330, looseness=1.0] (p3);
   
 \draw[dotted, ->, >=latex, line width=2pt] (I) to [out=270, in=150, looseness=1.0] (T4);

      \draw[->, >=latex, line width=2pt] (T5) to [out=00, in=180, looseness=1.0] (p4);
      
      
 \draw[ ->, >=latex, line width=2pt] (T7) to [out=0, in=180, looseness=1.0] (p10);
 
    \draw[->, >=latex, line width=2pt] (T4) to [out=0, in=180, looseness=1.0] (p18);
   
  \draw[->, >=latex, line width=2pt] (p18) to [out=0, in=120, looseness=1.0] (T12);
     
     \draw[->, >=latex, line width=2pt] (T4) to [out=210, in=70, looseness=1.0] (p12);

     \draw[->, >=latex, line width=2pt] (p12) to [out=30, in=210, looseness=1.0] (T5);

     \draw[->, >=latex, line width=2pt] (p12) to [out=300, in=150, looseness=1.0] (T7);
     
     \draw[->, >=latex, line width=2pt] (p13) to [out=120, in=330, looseness=1.0] (T5);

     \draw[->, >=latex, line width=2pt] (p14) to [out=330, in=100, looseness=1.0] (T7);
     
     \draw[->, >=latex, line width=2pt] (T7) to [out=30, in=240, looseness=1.0] (p13);

     \draw[->, >=latex, line width=2pt] (T5) to [out=270, in=60, looseness=1.0] (p14);

         \draw[ ->, >=latex, line width=2pt] (p4) to [out=270, in=90, looseness=1.0] (T9);

         \draw[ ->, >=latex, line width=2pt] (p10) to [out=30, in=210, looseness=1.0] (T9);

     \draw[->, >=latex, line width=2pt] (T9) to [out=300, in=270, looseness=1.0] (p12);

     \draw[->, >=latex, line width=2pt] (p4) to [out=30, in=190, looseness=1.0] (T13);
     
     \draw[->, >=latex, line width=2pt] (T13) to [out=330, in=00, looseness=1.0] (p15);
         
          \draw[ ->, >=latex, line width=2pt] (p19) to [out=60, in=270, looseness=1.0] (T14);
          
          \draw[ ->, >=latex, line width=2pt] (T14) to [out=30, in=00, looseness=1.0] (p20);
          
         \draw[ ->, >=latex, line width=2pt] (p20) to [out=180, in=150, looseness=1.0] (T14);

          \draw[ ->, >=latex, line width=2pt] (T9) to [out=0, in=180, looseness=1.0] (p19);

         \draw[ ->, >=latex, line width=2pt] (p20) to [out=60, in=90, looseness=1.0] (T12);
         
          \draw[ ->, >=latex, line width=2pt] (T12) to [out=270, in=165, looseness=1.0] (p20);

         \draw[ dotted, ->, >=latex, line width=2pt] (p20) to [out=45, in=150, looseness=1.0] (T13);

         \draw[ double, ->, >=latex, line width=2pt] (T1) to [out=240, in=120, looseness=1.0] (p12);

         \draw[ double, ->, >=latex, line width=2pt] (T1) to [out=40, in=60, looseness=1.0] (p18);

         \draw[ double, ->, >=latex, line width=2pt] (T1) to [out=50, in=150, looseness=1.0] (cloud1);

         \draw[ double, ->, >=latex, line width=2pt] (T3) to [out=40, in=200, looseness=1.0] (cloud1);

         \draw[ double, ->, >=latex, line width=2pt] (T4) to [out=20, in=210, looseness=1.0] (cloud1);

         \draw[->, >=latex, line width=2pt] (cloud1) to [out=270, in=90, looseness=1.0] (T13);

\end{tikzpicture}
}
\caption{Petri net for syzygy computation}
    \label{fig:petri-net-syz}
\end{figure}

\bibliographystyle{alpha} 
\bibliography{references}

\end{document}